\documentclass{article}%
\usepackage{amsmath}
\usepackage{amsfonts}
\usepackage{amssymb}
\usepackage{mathtools}
\usepackage{graphicx}
\usepackage{color}%
\providecommand{\U}[1]{\protect\rule{.1in}{.1in}}
\newtheorem{theorem}{Theorem}[section]
\newtheorem{corollary}[theorem]{Corollary}

\newtheorem{lemma}[theorem]{Lemma}
\newtheorem{proposition}[theorem]{Proposition}

\newtheorem{remark}[theorem]{Remark}

\def\Y_#1{\boldsymbol{Y}_{\!#1}}

\definecolor{darkred}{rgb}{0.9,0.1,0.1}
\newenvironment{proof}[1][Proof]{\noindent\textbf{#1.} }{\
  \rule{0.5em}{0.5em}}

\begin{document}

\title{ Randomized optimal stopping algorithms and their convergence analysis}
\author{Christian Bayer, Denis Belomestny, Paul Hager,\\Paolo Pigato, John Schoenmakers}
\maketitle

\begin{abstract}
In this paper we study randomized optimal stopping problems and consider corresponding forward and backward  Monte Carlo based optimisation algorithms. In
particular we prove the convergence of the proposed algorithms and derive
the corresponding convergence rates. \\[0.5em]{\scriptsize \textit{Keywords:}
randomized optimal stopping, convergence rates}\\[0.5em]%

\end{abstract}

\section{Introduction}

Optimal stopping problems play an important role in quantitative finance, as
some of the most liquid options are of American or Bermudan type, that is, they allow the
holder to exercise the option at any time before some terminal time or on a
finite, discrete set of exercise times, respectively. Mathematically, the price
of an American or Bermudan option is, hence, given as the solution of the
optimal stopping problem
\[
\sup_{\tau\in\mathcal{T}} \mathsf{E} Z_{\tau},
\]
where $Z_{t}$ denotes the discounted payoff to the holder of the option when
exercising at time $t$, and $\mathcal{T}$ denotes the set of all stopping
times taking values in $[0,T]$ in the American case and the set of stopping
times taking values in the set of exercise dates $\{t_{0}, \ldots, t_{J} \}$ in
the Bermudan case. Here $\mathsf{E}$ stands for the expectation w.r.t.~some risk
neutral measure. In this paper, we restrict ourselves to the Bermudan case,
either because the option under consideration is of Bermudan type or because
we have already discretized the American option in time.

\subsection{Background}

\label{sec:literature-review}

Due to the fundamental importance of optimal stopping in finance and
operations research, numerous numerical methods have been suggested. If the
dimension of the underlying driving process is high then deterministic methods
become inefficient. As a result most state-of-the-art methods are based on the
dynamic programming principle combined with Monte Carlo. This class includes
regression methods (local or global) \cite{J_LS2001,B1}, mesh methods \cite{BG}
and optimization based Monte Carlo algorithms
\cite{andersen1999simple,belomestny2011rates}. Other approaches include 
the quantization method \cite{Bally} and stochastic policy iteration \cite{J_KS2006}
for example.  
While the above methods aim at constructing (in general suboptimal) policies, hence lower estimations of the optimal stopping problem, the dual
approach  independently initiated by \cite{J_Rogers2002} and \cite{J_HK2004} 
has led to a stream of developments for computing upper bounds of the stopping problem (see, for
instance, \cite{BSbook} and the references therein).
\par
In this paper, we revisit optimization-based Monte Carlo (OPMC) algorithms and
extend them to the case of randomized stopping times. The idea behind OPMC methods
is to maximize a MC estimate of the associated value function over a family of
stopping policies thus approximating the early exercise region associated to
the optimal stopping problem rather than the value function. This idea goes back 
to \cite{A}. For a more general formulation see for instance \cite{J_Schoen2005}, Ch. 5.3, 
 and \cite{belomestny2011rates}  for a theoretical analysis.  
One problem of 
OPMC algorithms is that the corresponding loss functions are usually very
irregular, as was observed in \cite{belomestny2011rates} and \cite{BTW18}. In
order to obtain smooth optimization problems, the authors in
\cite{belomestny2016optimal} and \cite{BTW18} suggested to relax the problem
by randomizing the set of possible stopping times. For example, in the continuous exercise case, it was
suggested in \cite{BTW18} to stop at the first jump time of a Poisson process
with time and state dependent rate. The advantage of this approach is that the
resulting optimization problem becomes smooth. In general the solution of the
randomized optimal stopping problem coincides with the solution of the
standard optimal stopping problem, as earlier observed in
\cite{gyongy2008randomized}. 
\par
Let us also mention the recent works
\cite{becker2018deep,becker2019solving} that use deep neural networks to
solve optimal stopping problems numerically. These papers show very good
performance of deep neural networks for solving optimal stopping problems,
especially in high dimensions. However a complete convergence analysis of these
approaches is still missing.
A key issue in \cite{becker2018deep,becker2019solving} is 
a kind of smoothing of the functional representations of 
exercise boundaries or policies in order to make them suited for the standard gradient based 
optimization algorithms in the neural network based framework.  
In fact, we demonstrate that the randomized stopping provides a nice theoretical framework for such
smoothing techniques.   As such our results, in particular Corollary~\ref{Neur},  can be interpreted as a theoretical justification
of the neural network based methods in  \cite{becker2018deep,becker2019solving}. 
\par
Summing up, the contribution of this paper is twofold. On the one hand, we propose 
general OPMC methods that use randomized stopping times, instead of the
ordinary ones, thus leading to smooth optimization problems. On the other hand, we
provide a thorough convergence analysis of the proposed algorithms that
justify the use of randomized stopping times.
\par
The structure of the paper is as follows. In
Section~\ref{sec:rand-optim-stopp} we introduce the precise probabilistic
setting. In the following Section~\ref{sec:monte-carlo-optim} we introduce the
forward and the backward Monte Carlo methods. Convergence rates for both
methods are stated and proved in Section~\ref{conv}. In Section~\ref{numsec}
we describe a numerical implementation and present some numerical results for
the Bermudan max-call. Finally, there is an appendix with technical proofs
presented in Section~\ref{sec:proofs-auxil-results} and with a reminder on the
theory of empirical processes in Section~\ref{sec:some-results-from}.

\section{Randomized optimal stopping problems}

\label{sec:rand-optim-stopp}

Let $(\Omega,\mathcal{F},(\mathcal{F}_{j})_{j\geq0})$ be a given filtered
probability space, and $(\Omega_{0},\mathcal{B})$ be some auxiliary space that
is \textquotedblleft rich enough\textquotedblright\ in some sense. A
randomized stopping time $\tau$ is defined as a mapping from $\widetilde
{\Omega}:=\Omega\times\Omega_{0}$ to $\mathbb{N}$ (including $0$) that is
measurable with respect to the $\sigma$-field $\widetilde{\mathcal{F}}%
:=\sigma\left\{  F\times B:F\in\mathcal{F},\text{ }B\in\mathcal{B}\right\}  ,$
and satisfies
\[
\{\tau\leq j\}\in\widetilde{\mathcal{F}}_{j}:=\sigma\left\{  F\times
B:F\in\mathcal{F}_{j},\text{ }B\in\mathcal{B}\right\}  ,\text{ \ \ }%
j\in\mathbb{N}.
\]
While abusing notation a bit, $\mathcal{F}$ and $\mathcal{F}_{j}$ are
identified with $\sigma\left\{  F\times\Omega_{0}:F\in\mathcal{F}\right\}
\subset\widetilde{\mathcal{F}}$ and $\sigma\left\{  F\times\Omega_{0}%
:F\in\mathcal{F}_{j}\right\}  \subset\widetilde{\mathcal{F}}_{j},$
respectively. Let further $\mathsf{P}$ be a given probability measure on
$(\Omega,\mathcal{F}),$ and $\widetilde{\mathsf{P}}$ be its extension to $(\widetilde
{\Omega},\widetilde{\mathcal{F}})$ in the sense that%
\[
\widetilde{\mathsf{P}}\left(  \Omega_{0}\times F\right)  =\mathsf{P}\left(
F\right)  \text{ \ \ for all \ \ }F\in\mathcal{F}.
\]
In this setup we may think of $\mathsf{P}$ as the measure governing the
dynamics of some given adapted nonnegative reward process $(Z_{j})_{j\geq0}$
on $(\Omega,\mathcal{F},(\mathcal{F}_{j})_{j\geq0}).$ We then may write (with
$\mathsf{E}$ denoting the expectation with respect to the \textquotedblleft
overall measure\textquotedblright\ $\widetilde{\mathsf{P}}$)
\begin{equation}
\mathsf{E}\left[  Z_{\tau}\right]  =\mathsf{E}\left[  \sum_{j=0}^{\infty}%
Z_{j}p_{j}\right]  \label{eq:exp-rand}%
\end{equation}
with%
\[
p_{j}:=\mathsf{E}\left[  \left.  1_{\left\{  \tau=j\right\}  }\right\vert
\mathcal{F}_{j}\right]  =\widetilde{\mathsf{P}}\left(  \left.  \tau
=j\right\vert \mathcal{F}_{j}\right)  .
\]
Hence the sequence of nonnegative random variables $p_{0},p_{1},\ldots,$ is
adapted to $(\mathcal{F}_{j})_{j\geq0}$ and satisfies%
\[
\sum_{j=0}^{\infty}p_{j}=\sum_{j=0}^{\infty}\mathsf{E}\left[  \left.
1_{\left\{  \tau=j\right\}  }\right\vert \mathcal{F}\right]  =1\text{ \ a.s.}%
\]
In this paper we shall study discrete time optimal stopping problems of the
form
\begin{equation}
Y_{j}^{\star}=\sup_{\tau\in\widetilde{\mathcal{T}}[j,J]}\mathsf{E}[Z_{\tau
}|\mathcal{F}_{j}],\quad j=0,\ldots,J, \label{eq:opt-st-j}%
\end{equation}
where $\widetilde{\mathcal{T}}[j,J]$ is the set of randomized stopping times
taking values in $\{j,\ldots,J\}.$ It is well-known (see
\cite{belomestny2016optimal} and \cite{gyongy2008randomized}) that there
exists a family of ordinary stopping times $\tau_{j}^{\star},$ $j=0,\ldots,J,$
solving (\ref{eq:opt-st-j}) that satisfies the so-called consistency property
$\tau_{j}^{\star}>j\Longrightarrow\tau_{j}^{\star}=\tau_{j+1}^{\star}.$ That
is, at the same time,%
\[
Y_{j}^{\star}=\sup_{\tau\in\mathcal{T}[j,J]}\mathsf{E}[Z_{\tau}|\mathcal{F}%
_{j}],\quad j=0,\ldots,J
\]
where $\mathcal{T}[j,J]$ is the set of the (usual) $\mathcal{F}$-stopping
times. Studying \eqref{eq:opt-st-j} over a larger class of stopping times is
motivated by the fact that the set of randomized stopping times is convex, see
\cite{belomestny2016optimal}.
\par
From now on we consider the Markovian case with $Z_{j}=G_{j}(X_{j}),$ where
$(X_j)_{j\geq 0}$ is a Markov chain on $(\Omega,\mathcal{F},(\mathcal{F}_{j})_{j=0}^{J})$
with values in $\mathbb{R}^{d}$ and $(G_{j})_{j\geq 0}$ is a sequence of $\mathbb{R}%
^{d}\rightarrow\mathbb{R}_{+}$ functions. We also shall deal with consistent
families of randomized stopping times $(\tau_{j})_{j=0}^{J}$ satisfying
$j\leq\tau_{j}\leq J$ with $\tau_{J}=J$ and $\tau_{j}>j\Longrightarrow\tau
_{j}=\tau_{j+1}.$ Such a consistent family $(\tau_{k}),$ together with a
corresponding family of conditional exercise probabilities
\begin{equation}
p_{k,j}:=\mathsf{E}\left[  \left.  1_{\left\{  \tau_{k}=j\right\}
}\right\vert \mathcal{F}_{j}\right]  ,\quad j=k,\ldots,J, \label{conpr}%
\end{equation}
may be constructed by backward induction in the following way. We start with
$\tau_{J}=J$ almost surely and set $p_{J,J}=1$ reflecting the fact that, since $Z\geq0$
by assumption, stopping at $J$ is optimal provided one did not stop before
$J.$ Assume that $\tau_{k}$ with $k\leq\tau_{k}\leq J,$ and $p_{k,j},$
$j=k,\ldots,J,$ satisfying (\ref{conpr}), are already defined for some $k$
with $0<k\leq J.$

Next take some uniformly distributed random variable $\mathcal{U}\sim U[0,1],$
independent from $\mathcal{F}$ and $\tau_{k},$ and a function $h_{{k-1}}%
\in\mathcal{H}$ with $\mathcal{H}$ being a class of functions mapping
$\mathbb{R}^{d}$ to $[0,1].$ We then set
\begin{align*}
\tau_{k-1} =
\begin{cases}
k-1, & \mathcal{U}<h_{{k-1}}(X_{k-1}),\\
\tau_{k}, & \mathcal{U}\geq h_{{k-1}}(X_{k-1})
\end{cases}
\end{align*}
and
\begin{align}
\label{eq:p-recu}p_{k-1,j} =
\begin{cases}
h_{{k-1}}(X_{k-1}) & j=k-1\\
(1-h_{k-1}(X_{k-1}))\,p_{k,j} & j\geq k.
\end{cases}
\end{align}
Obviously, we then have for $j\geq k,$
\begin{align*}
\mathsf{E}\left[  \left.  1_{\left\{  \tau_{k-1}=j\right\}  }\right\vert
\mathcal{F}_{j}\right]   &  =\mathsf{E}\left[  1_{\left\{  \tau_{k}=j\right\}
}\mathsf{E}\left[  \left.  1_{\mathcal{U}\geq h_{{k-1}}(X_{k-1})}\right\vert
\mathcal{F}_{j}\vee\tau_{k}\right]  |\mathcal{F}_{j}\right] \\
&  =(1-h_{k-1}(X_{k-1}))\mathsf{E}\left[  \left.  1_{\left\{  \tau
_{k}=j\right\}  }\right\vert \mathcal{F}_{j}\right]  \,=p_{k-1,j}.
\end{align*}
That is, (\ref{conpr}) with $k$ replaced by $k-1$ is fulfilled.

It is immediately seen that, by the above construction, the thus obtained
family of randomized stopping times $(\tau_{k})_{k=0}^{J}$ is consistent, and
that%
\begin{equation}
\mathsf{E}[Z_{\tau_{k}}|\mathcal{F}_{k}]=\mathsf{E}\left[  \left.  \sum
_{j=k}^{\infty}Z_{j}p_{k,j}\right|  \mathcal{F}_{k}\right]  \label{randk}%
\end{equation}
with $h_{J}\equiv1$ by definition, and where%
\begin{equation}
p_{k,j}=h_{{j}}(X_{j})\prod_{l=k}^{j-1}(1-h_{{l}}(X_{l})),\text{
\ \ }j=k,\ldots,J. \label{rep}%
\end{equation}
Hence each (conditional) probability $p_{k,j}$ is a function of $X_{k}%
,\ldots,X_{j}$ by construction, and so in particular it is measurable with
respect to the $\sigma$- algebra $\mathcal{F}_{j}.$ In view of
\eqref{eq:exp-rand}, \eqref{eq:opt-st-j}, and (\ref{rep}), we now consider the
following optimization problems
\begin{equation}
\overline{Y}_{j}=\sup_{\boldsymbol{h}\in\mathcal{H}^{J-j}}\mathsf{E}\left[
\left.  \sum_{l=j}^{J}Z_{l}h_{{l}}(X_{l})\prod_{r=j}^{l-1}(1-h_{{r}}%
(X_{r}))\right\vert \mathcal{F}_{j}\right]  ,\quad j=0,\ldots,J-1,
\label{eq:ros-app}%
\end{equation}
where  empty products are equal to $1$ by definition, and the
supremum is taken over vector functions $\boldsymbol{h}=(h_{0},\ldots,
h_{J-1})\in\mathcal{H}^{J-j}.$ It is well known, that the optimal process
(Snell envelope) $(Y_{j}^{\star})$ can be attained by using indicator
functions $(h_{j})$ of the form $h_{j}(x)=1_{S_{j}^{\star}}(x)$ in
\eqref{eq:ros-app}, where the stopping regions $(S_{j}^{\star})$ have the
following characterization
\[
S_{j}^{\star}=\{x\in\mathbb{R}^{d}:G_{j}(x)\geq C_{j}(x)\},\quad
C_{j}(x)=\mathsf{E}[Y_{j+1}^{\star}|X_{j}=x],\quad j=0,\ldots,J-1,
\]
with $S_{J}^{\star}=\mathbb{R}^{d}$ by definition. A family of optimal
stopping times $(\tau_{j}^{\star})_{j=0,...,J}$ solving \eqref{eq:opt-st-j}
can then be defined as a family of first entry times
\begin{equation}
\tau_{j}^{\star}=\min\{j\leq l\leq J:X_{l}\in S_{l}^{\star}\},\quad
j=0,\ldots,J. \label{eq:tau-star}%
\end{equation}
Note that this definition implies that the family $(\tau_{j}^{\star}%
)_{j=0}^{J}$ is consistent.

\section{Monte Carlo optimization algorithms}

\label{sec:monte-carlo-optim}

We now propose two Monte Carlo optimization algorithms for estimating
$Y_{0}^{\star}$ in (\ref{eq:opt-st-j}). The first one (\textit{forward
approach}) is based on simultaneous optimization of a Monte Carlo estimate for
\eqref{eq:ros-app} over the exercise probability functions $h_{0},\ldots
,h_{J},$ whereas in the second approach these functions are estimated step by
step backwardly from $h_{J}$ down to $h_{0}$ based on (\ref{eq:p-recu}). The
latter procedure is referred to as the \textit{backward approach.}

\subsection{Forward approach}

\label{FA}

Let us consider the empirical counterpart of the optimization problem
\eqref{eq:ros-app} at time $j=0.$ To this end we generate a set of independent
trajectories of the chain $(X_{j}):$
\[
\mathcal{D}_{M}:=\left\{  (X_{0}^{(m)},\ldots,X_{J}^{(m)}),\quad
m=1,\ldots,M\right\}
\]
and consider the optimization problem
\begin{equation}
\sup_{\boldsymbol{h}\in\mathcal{H}^{J}}\left\{\frac{1}{M}\sum_{m=1}^{M}\left[
\sum_{l=0}^{J}G_{l}(X_{l}^{(m)})h_{{l}}(X_{l}^{(m)})\prod_{r=0}^{l-1}%
(1-h_{{r}}(X_{r}^{(m)}))\right] \right\} . \label{O1}%
\end{equation}
Let $\boldsymbol{h}_{M}$ be one of its solutions. Next we generate new $N$ independent
paths of the chain $(X_{j})_{j=0}^{J}$ and build an estimate for the optimal
value $Y_{0}^{\star}$ as
\begin{equation}
Y_{M,N}=\frac{1}{N}\sum_{n=1}^{N}\left[  \sum_{l=0}^{J}G_{l}(X_{l}%
^{(n)})h_{M,l}(X_{l}^{(n)})\prod_{r=0}^{l-1}(1-h_{M,r}(X_{r}^{(n)}))\right]
\label{RO1}%
\end{equation}
Note that the estimate $Y_{M,N}$ is low biased, that is, $\mathsf{E}%
[Y_{M,N}|\mathcal{D}_{M}]\leq Y_{0}^{\star}.$ The algorithms based on
(\ref{O1}) and (\ref{RO1}) are referred to as \textit{forward algorithms} in
contrast to the \textit{backward algorithms} described in the next section.

In Section~\ref{convf} we shall study the properties of the estimate $Y_{M,N}$
obtained by the forward approach. In particular we there show that $Y_{M,N}$
converges to $Y_{0}^{\star}$ as $N,M\rightarrow\infty,$ and moreover we derive
the corresponding convergence rates.

\subsection{Backward approach}

\label{back}

The forward approach in the previous sections can be rather costly especially
if $J$ is large, as it requires optimization over a product space
$\mathcal{H}^{J}.$ In this section we propose an alternative approximative
method which is based on a backward recursion. Fix again a class $\mathcal{H}$
of functions $h:$ $\mathbb{R}^{d}\rightarrow\lbrack0,1].$ We construct
estimates $\widehat{h}_{J},\ldots,\widehat{h}_{0}$ recursively using a set of
trajectories
\[
D_{M}:=\left\{  \left(  X_{0}^{(m)},X_{1}^{(m)},\ldots,X_{J}^{(m)}\right)
,\,m=1,\ldots,M\right\}  .
\]
We start with $\widehat{h}_{J}\equiv1.$ Suppose that $\widehat{h}_{k}%
,\ldots,\widehat{h}_{J}$ are already constructed, then define
\begin{equation}
\widehat{h}_{k-1}:=\underset{h\in\mathcal{H}}{\arg\sup}\,\widehat{Q}%
_{k-1}(h,\widehat{h}_{k}\ldots,\widehat{h}_{J}) \label{b1}%
\end{equation}
with
\begin{multline}
\widehat{Q}_{k-1}(h_{k-1},\ldots,h_{J}):=\\
=\frac{1}{M}\sum_{m=1}^{M}\left[  \sum_{j=k-1}^{J}G_{j}(X_{j}^{(m)}%
)h_{j}(X_{j}^{(m)})\prod_{l=k-1}^{j-1}(1-h_{l}(X_{l}^{(m)}))\right]
\label{b2}%
\end{multline}
in view of (\ref{eq:ros-app}).

\begin{remark}
Note that the optimal functions $h_{j}^{\star}(x)=1_{\{x\in S_{j}^{\star}\}},$
$j=0,\ldots,J-1,$ can be sequentially constructed via relations
\[
h_{k-1}^{\star}:=\underset{h\in\mathcal{H}}{\arg\sup}[Q_{k-1}(h,h_{k}^{\star
},\ldots,h_{J}^{\star})],\quad h_{J}^{\star}\equiv1,
\]
where
\begin{equation}
Q_{k-1}(h_{k-1},h_{k},\ldots,h_{J}):=\mathsf{E}\left[  \sum_{j=k-1}^{J}%
Z_{j}h_{j}(X_{j})\prod_{l=k-1}^{j-1}(1-h_{l}(X_{l}))\right]  , \label{qkexp}%
\end{equation}
(see also (\ref{randk}) and (\ref{rep})), provided that $h_{1}^{\star}%
,\ldots,h_{J}^{\star}\in\mathcal{H}.$ This fact was used in
\cite{becker2019solving} to construct approximations for $h_{j}^{\star},$
$j=0,\ldots, J,$ via neural networks. Although it might seem appealing to
consider classes of functions $\mathcal{H}:$ $\mathbb{R}^{d}\to\{0,1\},$ this
may lead to nonsmooth and nonconvex optimization problems. Here we present
general framework allowing us to balance between smoothness of the class
$\mathcal{H}$ and its ability to approximate $h_{j}^{\star},$ $j=0,\ldots, J,$
see Section~\ref{convb}.
\end{remark}

Working all the way back we thus end up with a sequence $\widehat{h}%
_{J},\ldots,\widehat{h}_{0}$ and, similar to (\ref{RO1}), may next obtain a
low-biased approximation $\widehat{Y}_{M,N}$ via an independent re-simulation
with sample size $N.$ By writing%
\begin{multline*}
\widehat{Q}_{k-1}(h_{k-1},\ldots,h_{J})=\\
\frac{1}{M}\sum_{m=1}^{M}\left(  G_{k-1}(X_{k-1}^{(m)})-\sum_{j=k}^{J}%
G_{j}(X_{j}^{(m)})h_{j}(X_{j}^{(m)})\prod_{l=k}^{j-1}(1-h_{l}(X_{l}%
^{(m)}))\right)  h_{k-1}(X_{k-1}^{(m)})\\
+\frac{1}{M}\sum_{m=1}^{M}\sum_{j=k}^{J}G_{j}(X_{j}^{(m)})h_{j}(X_{j}%
^{(m)})\prod_{l=k}^{j-1}(1-h_{l}(X_{l}^{(m)}))
\end{multline*}
we see that the backward step (\ref{b1})-(\ref{b2}) is equivalent to
\begin{align}
\widehat{h}_{k-1}  &  =\underset{h\in\mathcal{H}}{\arg\sup}\,\widehat{Q}%
_{k-1}(h,\widehat{h}_{k}\ldots,\widehat{h}_{J})\nonumber\\
&  =\underset{h\in\mathcal{H}}{\arg\sup}\sum_{m=1}^{M}h(X_{k-1}^{(m)}%
)\nonumber\\
&  \times\left(  G_{k-1}(X_{k-1}^{(m)})-\sum_{j=k}^{J}G_{j}(X_{j}%
^{(m)})\widehat{h}_{j}(X_{j}^{(m)})\prod_{l=k}^{j-1}(1-\widehat{h}_{l}%
(X_{l}^{(m)}))\right) \nonumber\\
&  =:\underset{h\in\mathcal{H}}{\arg\sup}\sum_{m=1}^{M}\xi_{k-1}%
^{(m)}h(X_{k-1}^{(m)}) \label{opt}%
\end{align}
Advantage of the backward algorithm is its computational efficiency: in each
step of the algorithm we need to perform optimization over a space
$\mathcal{H}$ and not over the product space $\mathcal{H}^{J}$ as in the
forward approach.

\section{Convergence analysis}

\label{conv} In this section we provide a convergence analysis of the
procedures introduced in Section~\ref{FA} and Section~\ref{back} respectively.

\subsection{Convergence analysis of the forward approach}

\label{convf}

We make the following assumptions.

\begin{description}
\item[(H)] Denote for any $\boldsymbol{h}_{1},\boldsymbol{h}_{2}\in
\mathcal{H}^{J},$
\begin{align*}
d_{X}(\boldsymbol{h}_{1},\boldsymbol{h}_{2}):=\sqrt{\mathsf{E}\left[  \left|
\sum_{j=0}^{J-1}\left|  h_{1,j}(X_{j})-h_{2,j}(X_{j})\right|  \prod
_{l=0}^{j-1}(1-h_{2,l}(X_{l}))\right|  ^{2}\right]  }%
\end{align*}
Assume that the class of functions $\mathcal{H} $ is such that
\begin{align}
\label{CA}\log[\mathcal{N}(\delta,\mathcal{H}^{J},d_{X})]\leq A\delta^{-\rho}%
\end{align}
for some constant $A>0 $, any $0<\delta<1 $ and some $\rho\in(0,2), $ where
$\mathcal{N}$ is the minimal number (covering number) of closed balls of
radius $\delta$ (w.r.t. $d_{X}$) needed to cover the class $\mathcal{H} $.

\item[(G)] Assume that all functions $G_{j}, $ $j=0,\ldots, J,$ are uniformly
bounded by a constant $C_{Z}.$

\item[(B)] Assume that the inequalities
\begin{align}
\label{BA} &  \mathsf{P}\bigl(|G_{j}(X_{j})-C_{j}(X_{j})]|\leq\delta\bigr)\leq
A_{0,j}\delta^{\alpha}, \quad\delta<\delta_{0}%
\end{align}
hold for some $\alpha>0 $, $A_{0,j}>0, $ $j=1,\ldots, J-1 $, and $\delta_{0}>0
$.
\end{description}

\begin{remark}
Note that
\begin{align*}
d_{X}(\boldsymbol{h}_{1},\boldsymbol{h}_{2})\leq\sum_{j=0}^{J-1} \|{h}%
_{1,j}-{h}_{2,j}\|_{L_{2}(P_{X_{j}})},
\end{align*}
where $P_{X_{i}}$ stands for the distribution of $X_{i}.$ Hence \eqref{CA}
holds if
\begin{equation}
\label{eq:entr-alter}\max_{j=0,\ldots,J-1}\log[\mathcal{N}(\delta
,\mathcal{H},L_{2}(P_{X_{j}}))]\leq\left(  \frac{A^{\prime}}{\delta}\right)
^{\rho}%
\end{equation}
for some constant $A^{\prime}>0.$
\end{remark}

\begin{theorem}
\label{thr:conv-forward} Assume that assumptions (\textbf{H}), (\textbf{G})
and (\textbf{B}) hold. Then for any $U>U_{0} $ and $M>M_{0} $ it holds with
probability at least $1-\delta,$
\begin{align}
\label{UBounds}0\leq Y^{\star}_{0}-\mathsf{E}[Y_{M,N}|\mathcal{D}_{M}]\leq
C\left(  \frac{\log^{2}(1/\delta)}{M}\right)  ^{\frac{1+\alpha}{2+\alpha
(1+\rho)}}%
\end{align}
with some constants $U_{0}>0 $, $M_{0}>0 $ and $C>0, $ provided that
\begin{align}
\label{eq:bias}0\leq Y_{0}^{\star}-\overline{Y}_{0}\leq DM^{-1/(1+\rho/2)},
\end{align}
for a constant $D>0,$ where $\overline{Y}_{0}$ is defined in \eqref{eq:ros-app}.
\end{theorem}

\begin{proof}
Denote
\begin{align*}
\mathcal{Q}(\boldsymbol{h}):=\mathsf{E}\left[  \sum_{j=0}^{J}Z_{j}
p_{j}(\boldsymbol{h}) \right]  , \quad\Delta(\boldsymbol{h}):=\mathcal{Q}%
(\boldsymbol{h}^{\star})-\mathcal{Q}(\boldsymbol{h})
\end{align*}
with
\begin{align*}
p_{j}(\boldsymbol{h}):=h_{j}(X_{j})\prod_{l=0}^{j-1}(1-h_{l}(X_{l})),\quad
j=0,\ldots,J.
\end{align*}
Define also $\Delta_{M}(\boldsymbol{h}):=\sqrt{M}(\mathcal{Q}_{M}%
(\boldsymbol{h})-\mathcal{Q}(\boldsymbol{h}))$ with
\begin{align*}
\mathcal{Q}_{M}(\boldsymbol{h}):=\frac{1}{M}\sum_{m=1}^{M}\left[  \sum
_{j=0}^{J}Z_{j}^{(m)}h_{j}(X_{j}^{(m)})\prod_{l=0}^{j-1}(1-h_{l}(X_{l}%
^{(m)}))\right]
\end{align*}
and put $\Delta_{M}(\boldsymbol{h}^{\prime},\boldsymbol{h}):=\Delta
_{M}(\boldsymbol{h}^{\prime})-\Delta_{M}(\boldsymbol{h}). $ Let $\overline
{\boldsymbol{h}}$ be one of the solutions of the optimization problem
$\sup_{\boldsymbol{h}\in\mathcal{H}^{J}}\mathcal{Q}(\boldsymbol{h}).$ Since
$\mathcal{Q}_{M}(\boldsymbol{h}_{M})\geq\mathcal{Q}_{M}(\overline
{\boldsymbol{h}}),$ we obviously have
\begin{align}
\label{PUB1}\Delta(\boldsymbol{h}_{M})  &  \leq\Delta(\overline{\boldsymbol{h}%
})+\frac{\left[  \Delta_{M}({\boldsymbol{h}}^{\star},\overline{\boldsymbol{h}%
}) +\Delta_{M}(\boldsymbol{h}_{M},{\boldsymbol{h}}^{\star}) \right]  }%
{\sqrt{M}}.
\end{align}
Set $\varepsilon_{M}=M^{-1/(2+\rho)} $ and derive
\begin{multline}
\label{BINEQ}\Delta(\boldsymbol{h}_{M})\leq\Delta(\overline{\boldsymbol{h}})+
\frac{2}{\sqrt{M}}\sup_{\boldsymbol{h}\in\mathcal{H}^{J}:\, \Delta
_{X}(\boldsymbol{h}^{\star},\boldsymbol{h})\leq\varepsilon_{M} } |\Delta
_{M}(\boldsymbol{h}^{\star},\boldsymbol{h})|\\
+2\frac{\Delta_{X}^{(1-\rho/2)}(\boldsymbol{h}^{\star},\boldsymbol{h}_{M}%
)}{\sqrt{M}}\, \sup_{\boldsymbol{h}\in\mathcal{H}^{J}:\, \Delta_{X}%
(\boldsymbol{h}^{\star},\boldsymbol{h})>\varepsilon_{M} } \left[
\frac{|\Delta_{M}(\boldsymbol{h}^{\star},\boldsymbol{h})|} {\Delta
_{X}^{(1-\rho/2)}(\boldsymbol{h}^{\star},\boldsymbol{h})} \right]  ,
\end{multline}
where
\begin{align*}
\Delta_{X}(\boldsymbol{h}_{1},\boldsymbol{h}_{2}):=\sqrt{\mathsf{E}\left[
\left|  \sum_{j=0}^{J}Z_{j}p_{j}(\boldsymbol{h}_{1})-\sum_{j=0}^{J}Z_{j}%
p_{j}(\boldsymbol{h}_{2})\right|  ^{2}\right]  }, \quad\boldsymbol{h}_{1},
\boldsymbol{h}_{2} \in\mathcal{H}^{J}.
\end{align*}
The reason behind splitting the r. h. s. of \eqref{PUB1} into two parts is
that the behavior of the empirical process $\Delta_{M}(\boldsymbol{h}%
^{*},\boldsymbol{h}) $ is different on the sets $\{\boldsymbol{h}%
\in\mathcal{H}^{J}:\, \Delta_{X}(\boldsymbol{h}^{*},\boldsymbol{h}%
)>\varepsilon_{M} \} $ and $\{\boldsymbol{h}\in\mathcal{H}^{J}:\, \Delta
_{X}(\boldsymbol{h}^{*},\boldsymbol{h})\leq\varepsilon_{M} \}. $ The following
lemma holds.

\begin{lemma}
\label{lem:Delta-d} It holds
\begin{align*}
\Delta_{X}(\boldsymbol{h}_{1},\boldsymbol{h}_{2})\leq C_{Z}d_{X}%
(\boldsymbol{h}_{1},\boldsymbol{h}_{2})
\end{align*}
for any $\boldsymbol{h}_{1},\boldsymbol{h}_{2}\in\mathcal{H}^{J}.$
\end{lemma}

Lemma~\ref{lem:Delta-d} immediately implies that
\begin{align}
\label{ENT_DG}\log\bigl(\mathcal{N}(\delta,\mathcal{H}^{J},\Delta
_{X})\bigr) \leq J\log\bigl(\mathcal{N}(\delta,\mathcal{H},d_{X})\bigr)
\end{align}
where $\mathcal{N}(\delta,\mathcal{S},d) $ is the covering number of a class
$\mathcal{S} $ w.r.t. the pseudo-distance $d.$ Hence due to the assumption
(\textbf{H}) we derive from \cite{van2000applications} that for any
$\boldsymbol{h}\in\mathcal{H}^{J} $ and any $U>U_{0}, $
\begin{align}
\label{INEQ1}\mathsf{P}\left(  \sup_{\boldsymbol{h}^{\prime}\in\mathcal{H}%
^{J},\,\Delta_{X}(\boldsymbol{h},\boldsymbol{h}^{\prime})\leq\varepsilon_{M}
}|\Delta_{M}(\boldsymbol{h},\boldsymbol{h}^{\prime})|>U\varepsilon_{M}
^{1-\rho/2} \right)  \leq C\exp\left(  -\frac{U}{\varepsilon_{M} ^{\rho}C^{2}}
\right)
\end{align}
and
\begin{align}
\label{INEQ2}\mathsf{P}\left(  \sup_{\boldsymbol{h}^{\prime}\in\mathcal{H}%
^{J},\,\Delta_{X}(\boldsymbol{h},\boldsymbol{h}^{\prime})>\varepsilon_{M}
}\frac{|\Delta_{M}(\boldsymbol{h},\boldsymbol{h}^{\prime})|}{\Delta^{1-\rho/2
}_{X}(\boldsymbol{h},\boldsymbol{h}^{\prime})}>U \right)  \leq C \exp(-U/C^{2}
),
\end{align}
\begin{align}
\label{INEQ3}\mathsf{P}\left(  \sup_{\boldsymbol{h}^{\prime}\in\mathcal{H}%
^{J}}|\Delta_{M}(\boldsymbol{h},\boldsymbol{h}^{\prime})|>z\sqrt{M} \right)
\leq C \exp(-Mz^{2}/C^{2}B )
\end{align}
with some constants $C>0, \, B>0 $ and $U_{0}>0.$ To simplify notations
denote
\begin{align*}
\mathcal{W}_{1,M}:=\sup_{\boldsymbol{h}\in\mathcal{H}^{J}:\, \Delta
_{X}(\boldsymbol{h}^{*},\boldsymbol{h})\leq\varepsilon_{M} } |\Delta
_{M}(\boldsymbol{h}^{*},\boldsymbol{h})|,\\
\mathcal{W}_{2,M}:=\sup_{\boldsymbol{h}\in\mathcal{H}^{J}:\, \Delta
_{X}(\boldsymbol{h}^{*},\boldsymbol{h})>\varepsilon_{M} } \frac{|\Delta
_{M}(\boldsymbol{h}^{*},\boldsymbol{h})|} {\Delta_{X}^{(1-\rho/2)}%
(\boldsymbol{h}^{*},\boldsymbol{h})}%
\end{align*}
and set $\mathcal{A}_{0}:=\{ \mathcal{W}_{1,M}\leq U\varepsilon_{M}
^{1-\rho/2} \} $ for some $U>U_{0}. $ Then the inequality \eqref{INEQ1} leads
to the estimate
\[
\mathsf{P}(\bar{\mathcal{A}}_{0})\leq C\exp(-U\varepsilon_{M} ^{-\rho}/C^{2}
).
\]
Furthermore, since $\Delta(\overline{\boldsymbol{h}})\leq D M^{-1/(1+\rho/2 )}
$ and $\varepsilon^{1-\rho/2}_{M}/\sqrt{M}=M^{-1/(1+\rho/2)} $, we get on
$\mathcal{A}_{0} $
\begin{align}
\label{BINEQ1}\Delta(\boldsymbol{h}_{M})  &  \leq C_{0}M^{-1/(1+\rho/2
)}+2\frac{\Delta_{X}^{(1-\rho/2)}(\boldsymbol{h}^{*},\boldsymbol{h}_{M}%
)}{\sqrt{M}}\,\mathcal{W}_{2,M}%
\end{align}
with $C_{0}=D+2U $. Now we need to find a bound for $\Delta_{X}(\boldsymbol{h}%
^{*},\boldsymbol{h}_{M}) $ in terms of $\Delta(\boldsymbol{h}_{M}). $ This is
exactly the place, where the condition \eqref{BA} is used. The following
proposition holds.

\begin{proposition}
\label{prop:DeltaX-Delta} Suppose that the assumption (\textbf{BA}) holds for
$\delta>0 $, then there exists a constant $c_{\alpha}>0 $ not depending on $J$
such that for all $\boldsymbol{h} \in\mathcal{H}^{J}$ it holds
\begin{align*}
\Delta_{X}(\boldsymbol{h}^{\star},\boldsymbol{h})\leq c_{\alpha}\sqrt{J}%
\Delta^{\alpha/(2(1+\alpha))}(\boldsymbol{h}),
\end{align*}
where $c_{\alpha}$ depends on $\alpha$ only.
\end{proposition}

Let us introduce a set
\begin{align*}
\mathcal{A}_{1}:=\left\{  \Delta(\boldsymbol{h}_{M})>C_{0}(1-\varkappa
)^{-1}M^{-1/(1+\rho/2)} \right\}
\end{align*}
for some $0<\varkappa<1. $ Thus we get on $\mathcal{A}_{0}\cap\mathcal{A}_{1}
$
\begin{align*}
\Delta(\boldsymbol{h}_{M})\leq C_{1}\frac{ \Delta^{\alpha(1-\rho
/2)/(2(1+\alpha))}(\boldsymbol{h}_{M})}{\varkappa\sqrt{M}}\,\mathcal{W}_{2,M},
\end{align*}
where the constant $C_{1} $ depends on $\alpha$ but not on $\rho. $ Therefore
\begin{align*}
\Delta(\boldsymbol{h}_{M})\leq(\varkappa/C_{1})^{-\nu} M^{-\nu/2}%
\mathcal{W}^{\nu}_{2,M}%
\end{align*}
with $\nu:=\frac{2(1+\alpha) }{2+\alpha(1+\rho/2 )}. $ Applying inequality
\eqref{INEQ2} to $\mathcal{W}^{\nu}_{2,M} $ and using the fact that $\nu
/2\leq1/(1+\rho/2) $ for all $0<\rho\leq2, $ we finally obtain the desired
bound for $\Delta(\boldsymbol{h}_{M}), $
\begin{multline*}
\mathsf{P}\left(  \left\{  \Delta(\boldsymbol{h}_{M})>(V/M)^{\nu/2}\right\}
\cap\mathcal{A}_{1} \right)  \leq\\
\mathsf{P}\left(  \left\{  \Delta(\boldsymbol{h}_{M})>(V/M)^{\nu/2} \right\}
\cap\mathcal{A}_{0}\cap\mathcal{A}_{1}\right)  + \mathsf{P}(\bar
{\mathcal{A}_{0}})\\
\leq C\exp(-\sqrt{V}/B)+C\exp\left(  -UM^{\rho/(2+\rho)}/C^{2} \right)
\end{multline*}
which holds for all $V>V_{0} $ and $M>M_{0} $ with some constant $B $
depending on $\rho$ and $\alpha. $
\end{proof}

\subsection{Convergence analysis of the backward approach}

In this section we study the properties of the backward algorithm and prove
its convergence. The following result holds. \label{convb}

\begin{proposition}
\label{errstep} For any $k>1$ and any $h_{k-1},\ldots,h_{J}\in\mathcal{H},$
one has that
\begin{multline*}
0\leq Q_{k-1}(h_{k-1}^{\star},\ldots,h_{J}^{\star})-Q_{k-1}(h_{k-1}%
,\ldots,h_{J})\leq Q_{k}(h_{k}^{\star},\ldots,h_{J}^{\star})-Q_{k}%
(h_{k},\ldots,h_{J})\\
+\mathsf{E}\left[  \left(  Z_{k-1}-C_{k-1}^{\star}\right)  \left(
h_{k-1}^{\star}(X_{k-1})-h_{k-1}(X_{k-1})\right)  \right]  .
\end{multline*}
Note that $Z_{k-1}-C_{k-1}^{\star}\geq0$ if $h_{k-1}^{\star}(X_{k-1})=1$ and
$Z_{k-1}-C_{k-1}^{\star}<0$ if $h_{k-1}^{\star}(X_{k-1})=0$ due to the dynamic
programming principle.
\end{proposition}

This implies the desired convergence.

\begin{theorem}
Assume (\textbf{G}) and suppose that
\begin{equation}
\label{eq:entr-alt}\max_{j=0,\ldots,J-1} \mathcal{N}(\delta,\mathcal{H}%
,L_{2}(P_{X_{j}}))\leq\left(  \frac{\mathcal{A}}{\delta}\right)  ^{V}%
\end{equation}
holds for some $V>0$ and $\mathcal{A}>0.$ Then with probability at least
$1-\delta,$ and $k=1,\ldots,J,$
\begin{multline}
\label{eq:back-bound}0\leq Q_{k-1}(h_{k-1}^{\star},\ldots,h_{J}^{\star
})-Q_{k-1}(\widehat{h}_{k-1},\ldots,\widehat{h}_{J})\lesssim J\sqrt{\frac{V
\log(J\mathcal{A})}{M}}+\\
+J\sqrt{\frac{\log(1/\delta)}{M}}+\sum_{l=k-1}^{J-1}\inf_{h\in\mathcal{H}%
}\mathsf{E}\left[  \left(  Z_{l}-C_{l}^{\star}\right)  \left(  h_{l}^{\star
}(X_{l})-h(X_{l})\right)  \right]  ,
\end{multline}

\end{theorem}

where $\lesssim$ stands for inequality up to a constant depending on $C_{Z}$.

\begin{remark}
\label{rem:appr} A simple inequality
\begin{align*}
\sum_{l=k-1}^{J-1}\inf_{h\in\mathcal{H}}\mathsf{E}\left[  \left(  Z_{l}%
-C_{l}^{\star}\right)  \left(  h_{l}^{\star}(X_{l})-h(X_{l})\right)  \right]
\leq C_{Z}\sum_{l=k-1}^{J-1} \inf_{h\in\mathcal{H}}\|{h}_{l}-{h}^{\star}%
_{l}\|_{L_{2}(P_{X_{l}})}%
\end{align*}
shows that we can chose class $\mathcal{H}$ to minimise $\inf_{h\in
\mathcal{H}}\|{h}-{h}^{\star}_{j}\|_{L_{2}(P_{X_{j}})}.$ Consider classes
$\mathcal{H}$ of the form:
\begin{align*}
\mathcal{H}_{n,r}(R):=\left\{  \sum\limits_{i = 1}^{n} a_{i}\psi(A_{i} x +
b),\, A_{i}\in\mathbb{R}^{r\times d},\, a_{i}\in\mathbb{R},\, b\in
\mathbb{R}^{r},\, \sum_{i=1}^{n}|a_{i}|\leq R \right\}
\end{align*}
where $\psi: \mathbb{R}^{r} \to\mathbb{R}$ is infinitely many times
differentiable in some open sphere in $\mathbb{R}^{r}$ and $r\leq d$. Then
according to Corollary~\ref{cor:ind-appr},
\begin{align*}
\inf_{h\in\mathcal{H}_{n,r}(R)}\|{h}-{h}^{\star}_{j}\|_{L_{2}(P_{X_{j}})}\leq
C(d,\delta)\, n^{-\frac{1}{2d}+\delta}%
\end{align*}
for arbitrary small $\delta>0$ and large enough $R>0,$ provided that each
measure $P_{X_{j}}$ is regular in the sense that
\begin{align}
\label{eq:reg-ps}P_{X_{j}}\bigl (S^{\star}_{j}\setminus(S^{\star}%
_{j})^{\varepsilon}\bigr)\leq a_{d} \,\varepsilon,
\end{align}
where for any set $A$ in $\mathbb{R}^{d}$ we denote by $A^{\varepsilon}$ the
set $A^{\varepsilon}:=\{x\in\mathbb{R}^{d}:\, \mathrm{dist}(x,A)\leq
\varepsilon\}.$ Moreover, it is not difficult to see that \eqref{eq:entr-alt}
holds for $\mathcal{H}_{n,r}$ with $V$ proportional to $n$ and {$\mathcal{A}$ proportional to $R,$ see Lemma~16.6 in \cite{gyorfi2002distribution}}.
\end{remark}

\begin{corollary}\label{Neur}
Remark~\ref{rem:appr} implies that the second term in \eqref{eq:back-bound}
converges to $0$ at the rate $O(n^{-1/(2d)+\delta})$ as $n\to\infty,$ where
$n$ is the number of neurons in the approximating neural network. On the other
hand, the constant $\mathcal{A}$ in \eqref{eq:entr-alt} is proportional to
$n,$ so that we get
\begin{align*}
Q_{k-1}(h_{k-1}^{\star},\ldots,h_{J}^{\star})-Q_{k-1}(\widehat{h}_{k-1}%
,\ldots,\widehat{h}_{J})\lesssim J\sqrt{\frac{n J\log(\mathcal{A})}{M}}+\\
+J\sqrt{\frac{\log(1/\delta)}{M}}+\frac{J}{n^{1/(2d)}}.
\end{align*}
By balancing two errors we arrive at the bound
\begin{align*}
0\leq Q_{k-1}(h_{k-1}^{\star},\ldots,h_{J}^{\star})-Q_{k-1}(\widehat{h}%
_{k-1},\ldots,\widehat{h}_{J})\lesssim J\left(  \sqrt{\frac{J}{M^{1+1/d}}}
+\sqrt{\frac{\log(1/\delta)}{M}}\right)
\end{align*}
with probability $1-\delta.$ In fact, this gives the overall error bounds for
the case of one layer neural networks based approximations.
\end{corollary}

\begin{proof}
Using (\ref{b1}), (\ref{b2}), (\ref{qkexp}), we have%
\begin{align}
0  &  \leq Q_{k-1}(h_{k-1}^{\star},\ldots,h_{J}^{\star})-Q_{k-1}(\widehat
{h}_{k-1},\ldots,\widehat{h}_{J})\nonumber\\
&  =\inf_{h_{k-1}\in\mathcal{H}}\left(  Q_{k-1}(h_{k-1}^{\star},\ldots
,h_{J}^{\star})-\widehat{Q}_{k-1}(h_{k-1},\widehat{h}_{k},\ldots,\widehat
{h}_{J})\right) \nonumber\\
&  +\widehat{Q}_{k-1}(\widehat{h}_{k-1},\ldots,\widehat{h}_{J})-Q_{k-1}%
(\widehat{h}_{k-1},\ldots,\widehat{h}_{J})\nonumber\\
&  \leq\inf_{h_{k-1}\in\mathcal{H}}\left(  Q_{k-1}(h_{k-1}^{\star}%
,\ldots,h_{J}^{\star})-Q_{k-1}(h_{k-1},\widehat{h}_{k},\ldots,\widehat{h}%
_{J})\right) \nonumber\\
&  +2\sup_{h_{k-1},\ldots,h_{J}\in\mathcal{H}}\left\vert Q_{k-1}(h_{k-1}%
,h_{k},\ldots,h_{J})-\widehat{Q}_{k-1}(h_{k-1},h_{k},\ldots,h_{J})\right\vert
\nonumber\\
&  \leq Q_{k}(h_{k}^{\star},\ldots,h_{J}^{\star})-Q_{k}(\widehat{h}_{k}%
,\ldots,\widehat{h}_{J})\nonumber\\
&  +\inf_{h_{k-1}\in\mathcal{H}}\mathsf{E}\left[  \left(  Z_{k-1}%
-C_{k-1}^{\star}\right)  \left(  h_{k-1}^{\star}(X_{k-1})-h_{k-1}%
(X_{k-1})\right)  \right] \\
&  +2\sup_{h_{k-1},\ldots,h_{J}\in\mathcal{H}}\left\vert Q_{k-1}(h_{k-1}%
,h_{k},\ldots,h_{J})-\widehat{Q}_{k-1}(h_{k-1},h_{k},\ldots,h_{J})\right\vert
, \label{lstep}%
\end{align}
where the last step follows from Proposition~\ref{errstep}. Set
\[
g_{\mathbf{h}}(X_{k-1},\ldots,X_{J})=\sum_{j=k-1}^{J}Z_{j}h_{j}(X_{j}%
)\prod_{l=k-1}^{j-1}(1-h_{l}(X_{l})),\quad\mathbf{h}=(h_{k-1},\ldots,h_{J}),
\]
then
\begin{multline*}
\widehat{Q}_{k-1}(h_{k-1},\ldots,h_{J})-Q_{k-1}(h_{k-1},\ldots,h_{J})\\
=\frac{1}{M}\sum_{m=1}^{M}\left\{  g_{\mathbf{h}}(X_{k-1}^{(m)},\ldots
,X_{J}^{(m)})-\mathsf{E}[g_{\mathbf{h}}(X_{k-1}^{(m)},\ldots,X_{J}%
^{(m)})]\right\}  .
\end{multline*}
Now consider the class $\mathcal{G}:=\{g_{\mathbf{h}},\,\mathbf{h}%
\in\mathcal{H}^{(J-k+1)}\}$ of uniformly bounded functions on $\mathbb{R}%
^{d(J-k+1)}.$ Indeed we have $|g_{\mathbf{h}}|\leq C_{Z}.$ Moreover
\begin{align}
\label{eq:entr-g}\mathcal{N}(\delta,\mathcal{G},L_{2}(P))\leq(\mathcal{A}%
/\delta)^{JV}%
\end{align}
under \eqref{eq:entr-alt}. Denote
\[
Z=\sqrt{M}\sup_{h_{k-1},\ldots,h_{J}\in\mathcal{H}}\left\vert Q_{k-1}%
(h_{k-1},h_{k},\ldots,h_{J})-\widehat{Q}_{k-1}(h_{k-1},h_{k},\ldots
,h_{J})\right\vert ,
\]
then the Talagrand inequality (see \cite{talagrand1994sharper} and
\cite{gine2002rates}) yields
\begin{align*}
\mathsf{P}(Z\geq\mathsf{E}[Z]+\sqrt{x(4C_{Z}\mathsf{E}[Z]+M)} +C_{Z} x/3)\leq
e^{-x},
\end{align*}
where
\begin{align*}
\mathsf{E}[Z]\leq\sqrt{M J \log(\mathcal{A} C_{Z})}%
\end{align*}
provided \eqref{eq:entr-g} holds.
\end{proof}

\section{Implementation of the Bermudan max-call}

\label{numsec}

In this section we implement  the pricing of the Bermudan max-call,  a benchmark example in the literature \cite{AndersenBroadie}. As underlying we take a
$d$-dimensional Black \& Scholes model, with log-price dynamics given by
\[
dX_{t}^{i}=\sigma dW^{i}_{t} + (r-\delta-\frac{\sigma^{2}}{2} )dt, \quad
i=1,\dots, d,
\]
where $X_{0}^{i}=0$ and $W^{i}$, $i=1,\dots, d,$ are independent Brownian
Motions. Parameters $\sigma, r, \delta$ represent respectively volatility,
interest, and dividend rate. The corresponding stock prices are given by
$S^{i}_{t}=S_{0}^{i} \exp(X^{i}_{t}),\,t\in[0,T]$. Our goal is the price of a
Bermudan max-call option, given by the following representation,
\[
\sup_{\tau\in\mathcal{T}} E [e^{-r\tau} \max_{i=1,\dots,d}(S^{i}_{\tau}%
-K)_{+}],
\]
where $\mathcal{T}$ is the set of stopping times in $\{t_{0}=0,t_{1},t_{2}%
,\dots,t_{J}=T\}$ adapted to the process $X=(X^{1},\dots,X^{d})$, and
$(\cdot)_{+}$ denotes the positive part.

\subsection{Backward approach}

In order to implement a randomized stopping strategy we need some suitable
parametric choice for $h$. For example we may take $h_{l}(x)$ to be the
composition of a polynomial in $x$ with the logistic function $\frac{e^{x}%
}{1+e^{x}}$ (cf. \cite[Framework 3.1]{becker2019solving} and \cite[Section
2.2]{becker2018deep}), i.e.
\begin{equation}
h_{l}(x)=h_{\theta_{l}}(x)=\frac{e^{pol_{\theta_{l}}^{g}(x)}}{1+e^{pol_{\theta
_{l}}^{g}}(x)},\text{ \ \ }l=0,...,J, \label{ero_pol}%
\end{equation}
where $pol_{\theta}^{g}(x)$ is a polynomial of degree $g$ in $x$ and $\theta$
is the vector of its coefficients. As another example we may compose
polynomials in $x$ with a Gumble type distribution function $\beta
(x):=1-\exp(-\exp(x)),$ i.e.%
\begin{equation}
h_{l}(x)=h_{\theta_{l}}(x)=1-\exp(-\exp(pol_{\theta_{l}}^{g}(x))),\text{
\ \ }l=0,...,J. \label{par1}%
\end{equation}
Both functions are smooth approximations of the indicator function.
Let us choose the latter and carry out the backward procedure in Section
\ref{back}. Assume that we have already determined parameters $\widehat{\theta}%
_{k},\dots,\widehat{\theta}_{J-1}$ (hence the corresponding functions
$\widehat{h}_{k},\dots,\widehat{h}_{J-1},$ $\widehat{h}_{J}=h_{J}^{\star}=1$).
Now, according to \eqref{opt}, the estimate of $\widehat{\theta}_{k-1}$ is
given by the maximization over $\theta$ of the function
\begin{equation}
\widehat{L}_{k-1}(h_{\theta},\widehat{h}_{k},\dots,\widehat{h}_{J})=\sum
_{m=1}^{M}\xi_{k-1}^{(m)}h_{\theta}(X_{k-1}^{(m)}), \label{loss}%
\end{equation}
where $\xi_{k-1}^{(m)}$ is as in \eqref{opt}. The corresponding gradient is given
by
\begin{equation}
\nabla_{\theta}\widehat{L}_{k-1}(h_{\theta},\widehat{h}_{k},\dots,\widehat
{h}_{J})=\sum_{m=1}^{M}\xi_{k-1}^{(m)}\nabla_{\theta}h_{\theta}(X_{k-1}%
^{(m)}). \label{grad_loss}%
\end{equation}
The parametric choice \eqref{par1} allows for the following representation of
the $\theta$-gradient. We may write straightforwardly%
\begin{equation}
\nabla_{\theta}h_{\theta}(x)=(1-h_{\theta}(x))\exp\left(  pol_{\theta}%
^{g}(x)\right)  \,\nabla_{\theta}pol_{\theta}^{g}(x) \label{der_ero_pol}%
\end{equation}
and since $\theta$ is the vector of the coefficients of $pol_{\theta}^{g}$,
the gradient $\nabla_{\theta}pol_{\theta}$ is the vector of monomials in $x$
of degree less or equal than $g$. Injecting this representation in
\eqref{grad_loss} we get an explicit expression for the gradient of the
objective function that we can feed into the optimization algorithm. In fact,
the catch of the randomization is the smoothness of $h_{\theta}$ in
\eqref{par1} with respect to $\theta.$ This in turn allows for gradient based
optimization procedures with explicitly given objective function and gradient.
However, a non-trivial issue is how to find the global maximum, at each step,
of the function $\widehat{L}$. This is also a well know issue in machine
learning, see for instance \cite[Section 2.6]{becker2019solving} or \cite[Section 2.3]{becker2018deep}. We do not dig into this question in the
present paper and just refer to \cite{becker2018deep,becker2019solving} for relevant literature.

\subsection{Forward approach}

We can alternatively write $h_j(X_j)=h(X_j,t_j)$ with $h(x,t)=h_\theta(x,t)$ a function depending on a parameter $\theta$ to be optimized. In this case we use the forward approach, since the backward induction cannot be used with this type of parametrization.

As an example (analogous to \eqref{par1}), we consider
\begin{equation}\label{eq:stopping_function_time}
h_\theta(x,t)=1-\exp(-\exp(pol_\theta^{g}(x,t))),
\end{equation}
with $pol_\theta^g$ polynomial of degree $g$ in $x$ and $t$, from which
\[
\begin{split}
\nabla_\theta h(x,t)&=\exp(-\exp(pol_\theta^{g}(x,t)))
\exp(pol_\theta^{g}(x,t)) \nabla_\theta pol_\theta^{g}(x,t) \\
&=
(1-h_\theta(x,t))
\exp(pol_\theta^{g}(x,t)) \nabla_\theta pol_\theta^{g}(x,t)
\end{split}
\]
As before, we want to maximize over $\theta$ the payoff
\[
\mathcal{P}=E\Big[ \sum_{j=1}^J Z_j(X_j) p_j^{\theta} (X_1,\dots,X_j)\Big].
\]
We have
\[
\nabla_\theta \mathcal{P}=E\Big[ \sum_{j=1}^J Z_j(X_j)  \nabla_\theta p_j^{\theta} (X_1,\dots,X_j)\Big]
\]
with $p_j^{{\theta}} (X_1,\dots,X_j)$ as in \eqref{rep}. Explicit computations give now
\[
\begin{split}
&\nabla_{\theta} p_j^{{\theta}} (X_1,\dots,X_j)
=
p_j^{{\theta}} (X_1,\dots,X_j) \\&
\bigg(
\frac{1}{h_{\theta} (X_j,t_j)}
\exp(pol_\theta^{g}(X_{j},t_j)) \nabla_\theta pol_\theta^{g}(X_{j},t_j)
-
 \sum_{l=1}^{j}
\exp(pol_\theta^{g}(X_l,t_l)) \nabla_\theta pol_\theta^{g}(X_l,t_l)
\bigg)
\end{split}\]
for $j=1,\dots, J$. We can compute $\nabla_\theta \mathcal{P}$
and use in the optimization this explicit expression for the gradient of the loss function.

\subsection{Numerical results}

We take parameters $r=0.05, \delta=0.1, \sigma=0.2$ (as in \cite{
AndersenBroadie,becker2019solving}). We first compute the stopping functions $h$ in \eqref{par1}
using the backward method with $M=10^{7}$ trajectories. The price is then
re-computed using $10^{7}$ independent trajectories. We compute each step in
the backward optimization as described in the previous section, using
polynomials of degree three in the case of two stocks (ten parameters
for each time in $t_{0},t_{1},\dots,t_{J}$).
We take $J=9$ and $T=3$, with $t_{i}=i/3, \, i=1,\dots,9$.

Then, we also implement the time-dependent stopping function in \eqref{eq:stopping_function_time} and optimize it using the forward method on the same example, this time using polynomials of degree four.
Results and
relative benchmark are given in Table~\ref{table:maxcall}. We report results obtained in \cite{becker2019solving} using neural networks (NN) and the confidence intervals (CI) given in \cite{AndersenBroadie}.

The experiments were also repeated with the alternative parametrization
(\ref{ero_pol}), with comparable numerical results.

\begin{table}[ptb]
\label{table:maxcall}%
\begin{center}
\begin{tabular}
{|c|c|c|c|c|c|}%
\hline
  $S_{0}$ & $K$ & {Backward, $g=3$}
& Forward, $g=4$ & NN price in
\cite{becker2019solving} & $95\%$ CI in \cite{AndersenBroadie}  \\\hline\hline
  $90$ & $100$ &
$8.072$
& $8.055$ & $8.072$ & $[8.053,8.082]$  \\\hline
  $100$ & $100$ &
$13.728$
&$13.882$ & $13.899$ & $[13.892,13.934]$  \\\hline
\end{tabular}
\end{center}
\caption{Bermudan max-call prices for Black-Scholes model, with $d=2, T=3, J=9$ and $r=0.05, \delta=0.1, \sigma=0.2$. }%
\end{table}

\appendix

\section{Proofs of auxiliary results}

\label{sec:proofs-auxil-results}

\subsection{Proof of Lemma~\ref{lem:Delta-d}}

Set
\[
T=\sum_{j=0}^{J}Z_{j}h_{1,j}(X_{j})\prod_{l=0}^{j-1}(1-h_{1,l}(X_{l}%
))-\sum_{j=0}^{J}Z_{j}h_{2,j}(X_{j})\prod_{l=0}^{j-1}(1-h_{2,l}(X_{l})).
\]
We have
\begin{align*}
T  &  =\sum_{j=0}^{J}Z_{j}\left(  h_{1,j}(X_{j})-h_{2,j}(X_{j})\right)
\prod_{l=0}^{j-1}(1-h_{2,l}(X_{l}))\\
&  +\sum_{j=0}^{J}Z_{j}h_{1,j}(X_{j})\left[  \prod_{l=0}^{j-1}(1-h_{1,l}%
(X_{l}))-\prod_{l=0}^{j-1}(1-h_{2,l}(X_{l}))\right]  .
\end{align*}
Due to the simple identity%

\[
\prod_{k=1}^{K}a_{k}-\prod_{k=1}^{K}b_{k}=\sum_{k=1}^{K}(a_{k}-b_{k}%
)\prod_{l=1}^{k-1}a_{l}\prod_{r=k+1}^{K}b_{r},
\]
we derive
\begin{multline*}
\prod_{l=0}^{j-1}(1-h_{1,l}(X_{l}))-\prod_{l=0}^{j-1}(1-h_{2,l}(X_{l}))=\\
\sum_{l=0}^{j-1}\left(  h_{2,l}(X_{l})-h_{1,l}(X_{l})\right)  \prod
_{s=0}^{l-1}(1-h_{2,s}(X_{s}))\prod_{m=l+1}^{j-1}(1-h_{1,m}(X_{m})).
\end{multline*}
Hence
\begin{align*}
T  &  =\sum_{j=0}^{J}Z_{j}\left(  h_{1,j}(X_{j})-h_{2,j}(X_{j})\right)
\prod_{l=0}^{j-1}(1-h_{2,l}(X_{l}))\\
&  +\sum_{j=0}^{J}Z_{j}h_{1,j}(X_{j})\sum_{l=r}^{j-1}\left(  h_{2,l}%
(X_{l})-h_{1,l}(X_{l})\right)  \prod_{s=0}^{l-1}(1-h_{1,s}(X_{s}%
))\prod_{m=l+1}^{j-1}(1-h_{2,m}(X_{m}))\\
&  =\sum_{j=0}^{J}Z_{j}\left(  h_{1,j}(X_{j})-h_{2,j}(X_{j})\right)
\prod_{l=0}^{j-1}(1-h_{2,l}(X_{l}))\\
&  +\sum_{l=0}^{J-1}\left(  h_{2,l}(X_{l})-h_{1,l}(X_{l})\right)  \prod
_{s=0}^{l-1}(1-h_{2,s}(X_{s}))\left\{  \sum_{j=l+1}^{J}Z_{j}h_{1,j}%
(X_{j})\prod_{m=l+1}^{j-1}(1-h_{1,m}(X_{m}))\right\}
\end{align*}
and
\[
|T|\leq C_{Z}\sum_{j=0}^{J-1}\left|  h_{1,j}(X_{j})-h_{2,j}(X_{j})\right|
\prod_{l=r}^{j-1}(1-h_{2,l}(X_{l})).
\]

\subsection{Proof of Proposition~\ref{prop:DeltaX-Delta}}

\begin{lemma}
Let $(\tau_{1,j})$ and $(\tau_{2,j})$ be two consistent families of randomized
stopping times, then
\end{lemma}

\[
\mathsf{E}_{r}\left[  Z_{\tau_{1,r}}-Z_{\tau_{2,r}}\right]  =\mathsf{E}%
_{r}\left[  \sum_{l=r}^{J-1}\left\{  Z_{l}-\mathsf{E}_{l}\left[
Z_{\tau_{1,l+1}}\right]  \right\}  \left(  q_{1,l}-q_{2,l}\right)  \prod
_{k=r}^{l-1}\left(  1-q_{2,k}\right)  \right]
\]
with $q_{i,j}=\widetilde{\mathsf{P}}\left(  \tau_{i,j}=j\right)  ,$ $i=1,2,$
and $\mathsf{E}_{r}[\cdot] =\mathsf{E}[\cdot|\mathcal{F}_{r}].$

\begin{proof}
We have
\begin{align*}
\mathsf{E}_{r}\left[  Z_{\tau_{1,r}}-Z_{\tau_{2,r}}\right]   &  =\left\{
Z_{r}-\mathsf{E}_{r}\left[  Z_{\tau_{1,r+1}}\right]  \right\}  \left(
\widetilde{\mathsf{P}}\left(  \tau_{1,r}=r\right)  -\widetilde{\mathsf{P}%
}\left(  \tau_{2,r}=r\right)  \right) \\
&  +\mathsf{E}_{r}\left[  \left(  \mathsf{E}_{r+1}\left[  Z_{\tau_{1,r+1}%
}-Z_{\tau_{2,r+1}}\right]  \right)  \widetilde{\mathsf{P}}\left(  \tau
_{2,r}>r\right)  \right] \\
&  =\left\{  Z_{r}-\mathsf{E}_{r}\left[  Z_{\tau_{1,r+1}}\right]  \right\}
\left(  q_{1,r}-q_{2,r}\right) \\
&  +\mathsf{E}_{r}\left[  \left(  \mathsf{E}_{r+1}\left[  Z_{\tau_{1,r+1}%
}-Z_{\tau_{2,r+1}}\right]  \right)  \left(  1-q_{2,r}\right)  \right]  .
\end{align*}
By denoting $\Delta_{r}=\mathsf{E}_{r}\left[  Z_{\tau_{1,r}}-Z_{\tau_{2,r}%
}\right]  ,$ we derive the following recurrent relation
\[
\Delta_{r}=\left\{  Z_{r}-\mathsf{E}_{r}\left[  Z_{\tau_{1,r+1}}\right]
\right\}  \left(  q_{1,r}-q_{2,r}\right)  +\mathsf{E}_{r}\left[  \Delta
_{r+1}\right]  \left(  1-q_{2,r}\right)  ,
\]
where all quantities with index $r$ are $\mathcal{F}_{r}-$measurable.
\end{proof}

Using the property \eqref{eq:tau-star} we derive an important corollary.

\begin{corollary}
\label{eq: cor-delta} It holds for any consistent family $(\tau_{r})_{r=0}%
^{J}$ of randomized stopping times,
\[
\mathsf{E}\left[  Z_{\tau_{r}^{\star}}-Z_{\tau_{r}}\right]  =\mathsf{E}\left[
\sum_{l=0}^{J-1}\left|  Z_{l}-\mathsf{E}_{l}\left[  Z_{\tau_{l+1}^{\star}%
}\right]  \right|  \left|  q_{l}^{\star}-q_{l}\right|  \prod_{k=r}%
^{l-1}\left(  1-q_{k}\right)  \right]  ,
\]
where $q_{l}=\widetilde{\mathsf{P}}\left(  \tau_{l}=l\right)  $ and $q^{\star
}_{l}=1_{\{\tau_{l}^{\star}=l\}}.$
\end{corollary}

Denote
\[
\mathcal{A}_{l}:=\left\{  \left|  G_{l}(X_{l})-C_{l}^{\star}(X_{l})\right|
>\delta\right\}  ,\quad l=0,\ldots,J-1,
\]
then
\begin{align*}
\Delta(\boldsymbol{h})  &  \geq\mathsf{E}\left[  \sum_{l=0}^{J-1}%
1_{\mathcal{A}_{l}}\left|  Z_{l}-\mathsf{E}_{l}\left[  Z_{\tau_{l+1}^{\star}%
}\right]  \right|  \left|  q_{l}^{\star}-q_{l}\right|  \prod_{k=0}%
^{l-1}\left(  1-q_{k}\right)  \right] \\
&  =\delta\left\{  \mathsf{E}\left[  \sum_{l=0}^{J-1}\left|  h_{l}^{\star
}(X_{l})-h_{l}(X_{l})\right|  \prod_{k=0}^{l-1}\left(  1-h_{k}(X_{k})\right)
\right]  -\sum_{l=0}^{J-1}\mathsf{P}\left(  \overline{\mathcal{A}_{l}}\right)
\right\}  .
\end{align*}
Due to (\textbf{B})
\[
\sum_{l=0}^{J-1}\mathsf{P}\left(  \overline{\mathcal{A}_{l}}\right)  \leq
A_{0}\delta^{\alpha},\quad A_{0}=\sum_{l=0}^{J-1} A_{0,l}
\]
and hence
\[
\Delta(\boldsymbol{h})\geq\delta\left\{  \frac{d_{X}^{2}(\boldsymbol{h}%
^{\star},\boldsymbol{h})}{J}-A_{0}\delta^{\alpha}\right\}  \geq\delta\left\{
\frac{\Delta_{X}^{2}(\boldsymbol{h}^{\star},\boldsymbol{h})}{JC_{Z}^{2}}%
-A_{0}\delta^{\alpha}\right\}
\]
due to Lemma~\ref{lem:Delta-d}. Taking maximum of the right hand side in
$\delta,$ we get
\[
\Delta_{X}(\boldsymbol{h}^{\star},\boldsymbol{h})\leq c_{\alpha}\sqrt{J}%
\Delta^{\alpha/(2(1+\alpha))}(\boldsymbol{h})
\]
for some constant $c_{\alpha}$ depending on $\alpha$ only.

\subsection{Proof of Proposition~\ref{errstep}}

By (\ref{qkexp})\ we may write,%
\begin{align*}
&  0\leq Q_{k-1}(h_{k-1}^{\star},h_{k}^{\star},\ldots,h_{J}^{\star}%
)-Q_{k-1}(h_{k-1},h_{k},\ldots,h_{J})\\
&  =\mathsf{E}\left[  Z_{k-1}h_{k-1}^{\star}(X_{k-1})-Z_{k-1}h_{k-1}%
(X_{k-1})\right] \\
&  +\mathsf{E}\left[  (1-h_{k-1}^{\star}(X_{k-1}))\mathsf{E}\left[  \left.
\sum_{j=k}^{J}Z_{j}h_{j}^{\star}(X_{j})\prod_{l=k}^{j-1}(1-h_{l}^{\star}%
(X_{l}))\right\vert \mathcal{F}_{k-1}\right]  \right] \\
&  -\mathsf{E}\left[  (1-h_{k-1}(X_{k-1}))\mathsf{E}\left[  \left.  \sum
_{j=k}^{J}Z_{j}h_{j}(X_{j})\prod_{l=k}^{j-1}(1-h_{l}(X_{l}))\right\vert
\mathcal{F}_{k-1}\right]  \right] \\
&  =\mathsf{E}\left[  \left(  Z_{k-1}-C_{k-1}^{\star}\right)  h_{k-1}^{\star
}(X_{k-1})+C_{k-1}^{\star}-Z_{k-1}h_{k-1}(X_{k-1})\right]  -Q_{k}(h_{k}%
,\ldots,h_{J})\\
&  +\mathsf{E}\left[  h_{k-1}(X_{k-1})\mathsf{E}\left[  \left.  \sum_{j=k}%
^{J}Z_{j}h_{j}(X_{j})\prod_{l=k}^{j-1}(1-h_{l}(X_{l}))\right\vert
\mathcal{F}_{k-1}\right]  \right] \\
&  =Q_{k}(h_{k}^{\star},h_{k}^{\star},\ldots,h_{J}^{\star})-Q_{k}(h_{k}%
,\ldots,h_{J})\\
&  +\mathsf{E}\left[  \left(  Z_{k-1}-C_{k-1}^{\star}\right)  \left(
h_{k-1}^{\star}(X_{k-1})-h_{k-1}(X_{k-1})\right)  \right] \\
&  +\mathsf{E}\left[  h_{k-1}(X_{k-1})\left(  \mathsf{E}\left[  \left.
\sum_{j=k}^{J}Z_{j}h_{j}(X_{j})\prod_{l=k}^{j-1}(1-h_{l}(X_{l}))\right\vert
\mathcal{F}_{k-1}\right]  -C_{k-1}^{\star}\right)  \right] \\
&  \leq Q_{k}(h_{k}^{\star},h_{k}^{\star},\ldots,h_{J}^{\star})-Q_{k}%
(h_{k},\ldots,h_{J})\\
&  +\mathsf{E}\left[  \left(  Z_{k-1}-C_{k-1}^{\star}\right)  \left(
h_{k-1}^{\star}(X_{k-1})-h_{k-1}(X_{k-1})\right)  \right]  .
\end{align*}

\section{Some auxiliary results}

\label{sec:some-results-from}

Let $\mathcal{X}\subset\mathbb{R}^{d}$ and let $\pi$ be a probability measure
on $\mathcal{X}$. We denote by $C(\mathcal{X})$ a set of all continuous
(possibly piecewise) functions on $\mathcal{X}$ and by $C^{s}(\mathcal{X})$
the set of all $s$-times continuously differentiable (possibly piecewise)
functions on $\mathcal{X}$. For a real-valued function $h$ on $\mathcal{X}%
\subset\mathbb{R}^{d}$ we write $\| h \|_{L^{p}(\pi)} = (\int_{\mathcal{X}}
|h(x)|^{p} \pi(x)dx)^{1/p}$ with $1\leq p < \infty$. The set of all functions
$h$ with $\| h \|_{L^{p}(\pi)}<\infty$ is denoted by $L^{p}(\pi)$.
%it is a normed space with the norm $\| \cdot \|_{L^p(\pi)}$.
If $\lambda$ is the Lebesgue measure, we write shortly $L^{p}$ instead of
$L^{p}(\lambda)$. The (real) Sobolev space is denoted by $W^{s,p}%
(\mathcal{X})$, i.e.,
\begin{align}
\label{eq:def_sobol}W^{s,p}(\mathcal{X}) := \left\{  u \in L^{p} : D^{\alpha}u
\in L^{p}, \quad\forall|\alpha| \leqslant s \right\}  ,
\end{align}
where $\alpha= (\alpha_{1},\ldots,\alpha_{d})$ is a multi-index with
$|\alpha|=\alpha_{1}+\ldots+\alpha_{d}$ and $D^{\alpha}$ stands for
differential operator of the form
\begin{align}
\label{eq:D-alpha}D^{\alpha} = \frac{\partial^{|\alpha|}}{\partial
x_{1}^{\alpha_{1}}\ldots\partial x_{d}^{\alpha_{d}}}.
\end{align}
Here all derivatives are understood in the weak sense. The Sobolev norm is
defined as
\[
\|u\|_{W^{s,p}(\mathcal{X})} = \sum\limits_{|\alpha|\le r}\|D^{\alpha
}u\|_{L^{p}}.
\]

\begin{theorem}
[Theorem 2.1 in \cite{mhaskar96}]\label{thm:approx} Let $1 \le r \le d, p\ge1,
n \ge1$ be integers, $\psi: \mathbb{R}^{r} \to\mathbb{R}$ be infinitely many
times differentiable in some open sphere in $\mathbb{R}^{r}$ and moreover,
there is $b \in\mathbb{R}^{r}$ in this sphere such that $D^{\alpha}\psi(b)
\neq0$ for all $\alpha$. Then there are $r \times d$ real matrices
$\{A_{j}\}_{j = 1}^{n}$ with the following property. For any $f \in
W^{s,p}(\mathcal{X})$ with $s\geq1$ there are coefficients $a_{j}(f)$
\[
\left\|  f - \sum\limits_{i = 1}^{n} a_{i}(f)\psi(A_{i}(\cdot) + b)\right\|
_{L^{p}(\mathcal{X})} \le\frac{c\|f\|_{W^{s,p}(\mathcal{X})}}{n^{s/d}},
\]
where $c$ is an absolute constant. Moreover, $a_{j}$ are continuous linear
functionals on $W^{s,p}(\mathcal{X})$ with $\sum_{j=1}^{n} |a_{j}|\leq C$ and
$C$ depending on $\|f\|_{L^{p}(\mathcal{X})}.$
\end{theorem}

For any set $A\subset\mathbb{R}^{d}$ let
\[
A^{\varepsilon}=\{x\in\mathbb{R}^{d}: \rho_{A}(x)\leq\varepsilon\}, \quad
\rho_{A}(x)=\inf_{y\in A}|x-y|.
\]

\begin{lemma}
\label{lem:ind-appr} Let a set $A\subset\mathbb{R}^{d}$ be convex. Then for
any $\varepsilon>0$ there exists a infinitely differentiable function
$\varphi_{A}$ with $0\leq\phi\leq1,$ such that
\begin{align*}
\varphi_{A}(x)=
\begin{cases}
1, & x\in A,\\
0, & x\in\mathbb{R}^{d}\setminus A^{\varepsilon}%
\end{cases}
\end{align*}
and for any multiindex $\alpha= (\alpha_{1},\ldots,\alpha_{d})$
\begin{align*}
|D^{\alpha}\varphi_{A}(x)|\leq\frac{C_{\alpha}}{\varepsilon^{|\alpha|}},\quad
x\in\mathbb{R}^{d}%
\end{align*}
with a constant $C_{\alpha}>0.$
\end{lemma}

\begin{corollary}
\label{cor:ind-appr} Let $S\subseteq\mathbb{R}^{d}$ be convex and let function
$\psi$ satisfy the conditions of Theorem~\ref{thm:approx}. Then for any fixed
$s>d,$ there are $r \times d$ real matrices $\{A_{j}\}_{j = 1}^{n}$ and $b
\in\mathbb{R}^{r}$ with the following property
\[
\left\|  1_{S}(\cdot)- \sum\limits_{i = 1}^{n} a_{i}(S)\psi(A_{i}(\cdot) +
b)\right\|  _{L_{2}(\pi)} \le\frac{C_{s}}{\varepsilon^{s} n^{s/d}}+\sqrt
{\pi(S\setminus S^{\varepsilon})}
\]
for some constant $C_{s}>0$ and some real numbers $a_{1},\ldots,a_{n}$
depending on $S$ such that $\sum_{i=1}^{n}|a_{i}| \leq Q$ where $Q$ is an
absolute constant.
\end{corollary}

\begin{proof}
Due to Lemma~\ref{lem:ind-appr}, there is an infinitely smooth function
$\phi_{S}$ such that
\begin{align*}
\sup_{x\in\mathbb{R}^{d}} |1_{S}(x)-\phi_{S}(x)|\leq\pi(S\setminus
S^{\varepsilon}).
\end{align*}
According to Theorem~\ref{thm:approx}, we have with $p=\infty$
\begin{align*}
\sup_{x\in\mathbb{R}^{d}}\left|  \varphi_{S}(x)-\sum\limits_{i=1}^{n}%
a_{i}(S)\psi(A_{i}x+b)\right|  \le\frac{C_{s}}{\varepsilon^{s}n^{s/d}}%
\end{align*}
for some matrices $r \times d$ real matrices $\{A_{j}\}_{j = 1}^{n}$ and real
numbers $a_{1},\ldots,a_{n}$ depending on $S.$
\end{proof}

\bibliographystyle{amsplain}
\bibliography{rand-stop_new}

\end{document}